\documentclass[12pt]{amsart}
\usepackage{amssymb}
\usepackage[all]{xypic}

\theoremstyle{plain}
\newtheorem{theorem}{Theorem}[section]
\newtheorem{lemma}[theorem]{Lemma}

\theoremstyle{remark}

\newtheorem{remark}[theorem]{Remark}

\newtheorem*{note*}{Note}
\newtheorem*{remark*}{Remark}
\newtheorem*{example*}{Example}

\theoremstyle{definition}
\newtheorem*{definition*}{Definition}

\newcommand{\Z}{\mathbb{Z}}

\newcommand{\Q}{\mathbb{Q}}

\newcommand{\Gal}{\mathrm{Gal}}

\newcommand{\Cl}{\mathrm{Cl}}

\newcommand{\frakp}{\mathfrak{p}}

\newcommand{\K}{\mathrm{K}}

\title[On totally real Hilbert-Speiser Fields of type $C_{p}$]{On totally real Hilbert-Speiser Fields of type $C_{p}$}

\author{Cornelius Greither}
\address{Cornelius Greither\\
Fakult\"at f\"ur Informatik\\
Institut f\"ur theoretische Informatik und Mathematik\\
Universit\"at der Bundeswehr M\"unchen\\
85577 Neubiberg\\
Germany}
\email{cornelius.greither@unibw.de}

\author{Henri Johnston}
\address{Henri Johnston\\ 
St. Hugh's College\\
St. Margaret's Road\\
Oxford OX2 6LE\\
U.K.
}
\email{henri@maths.ox.ac.uk}
\urladdr{http://people.maths.ox.ac.uk/henri}

\thanks{Johnston was partially supported by a grant from the
 Deutscher Akademischer Austausch Dienst.}
\subjclass[2000]{Primary 11R33}
\keywords{Galois module structure, normal integral basis, Hilbert-Speiser field}
\date{24th February 2009}

\begin{document}

\maketitle

\begin{abstract} 
Let $G$ be a finite abelian group. A number field $K$ is called a
Hilbert-Speiser field of type $G$ if every tame $G$-Galois extension $L/K$
has a normal integral basis, i.e.,
the ring of integers $\mathcal{O}_L$ is free as an $\mathcal{O}_KG$-module. 
Let $C_{p}$ denote the cyclic group of prime order $p$. 
We show that if $p \geq 7$ (or $p=5$ and extra conditions are met) and $K$ is totally 
real with $K/\Q$ ramified at $p$, then $K$ is not Hilbert-Speiser of type $C_{p}$.
\end{abstract}

\section{Introduction} 

Let $L/K$ be a finite Galois extension of number fields with Galois group $G$.
Then by a theorem of Noether it is well known that the ring of integers 
$\mathcal{O}_{L}$ is a projective module over the group ring $\mathcal{O}_KG$ 
if and only if $L/K$ is tamely ramified. If $\mathcal{O}_{L}$ is in fact free (necessarily
of rank $1$) over $\mathcal{O}_KG$, then $L/K$ is said to have a 
\emph{normal integral basis}.

A number field $K$ is called a \emph{Hilbert-Speiser field} if every finite abelian tamely
ramified extension $L/K$ has a normal integral basis. The celebrated Hilbert-Speiser 
Theorem says that $\mathbb{Q}$ is such a field, and the main result of \cite{grrs} is that 
$\mathbb{Q}$ is in fact the only such field.  By fixing a finite abelian group $G$ one can consider a finer problem: given a number field $K$, does every tame $G$-Galois extension $L/K$ have a normal integral basis? If so, $K$ is said to be a \emph{Hilbert-Speiser field of type $G$}. The simplest case to consider is when $G=C_{p}$, the cyclic group of prime order $p$.
This has been studied, for instance, in  \cite{carter-NIB-quad-cubic}, \cite{carter-erratum},
\cite{herreng-HS}, \cite{ichimura-kummer-prime-V}, \cite{ichimura-NIB-rayclassgroups},
\cite{Ichimura-HS-p}, \cite{Ichimura-imag-quad} and \cite{Ichimura-Sumida-Takahashi-HS-imag-quad}. We continue the investigation of this case by establishing the following result, the proof of which is based on a detailed analysis of locally free class groups and ramification indices.

\begin{theorem}\label{real-not-HS}
Let $K$ be a totally real number field and let $p\geq5$ be prime.
Suppose that $K/\Q$ is ramified at $p$. If $p=5$ and $[K(\zeta_{5}):K]=2$, 
assume further that there exists a prime $\mathfrak{p}$ of $K$ above $p$ such that the
ramification index of $\mathfrak{p}$ in $K/\Q$ is at least $3$. Then $K$ is not 
Hilbert-Speiser of type $C_{p}$.
\end{theorem}

\begin{remark}
Some extra conditions in the case $p=5$ and $[K(\zeta_{5}):K]=2$  are required
because, for example, as noted in \cite[Remark 1]{Ichimura-HS-p}, $K=\Q(\sqrt{5})$ is in fact Hilbert-Speiser of type $C_{5}$.
\end{remark}

Theorem \ref{real-not-HS} can be seen as an analogue of the following result of Herreng
(see \cite[\S 3]{herreng-HS}). The authors are grateful to Nigel P. Byott for pointing out that 
the original hypothesis that $K/\Q$ is Galois can be weakened as below.

\begin{theorem}[Herreng]\label{imag-not-HS}
Let $K$ be a totally imaginary number field and let $p$ be an odd prime.
Suppose that every prime $\frakp$ of $K$ above $p$ is ramified. If
\begin{enumerate}
\item $p > [K:\Q]$, or
\item $p \geq 5$ and $\zeta_{p} \in K$, or
\item $p \geq 7$ and the ramification index in $K/\Q$ of every prime $\frakp$ of $K$ above $p$
is at least $3$,
\end{enumerate}
then $K$ is not Hilbert-Speiser of type $C_{p}$.
\end{theorem}

Combining Theorems \ref{real-not-HS} and \ref{imag-not-HS}(a) we immediately obtain the 
following result, which in many (but not all) respects is a significant sharpening of 
\cite[Theorems 1 and 2]{Ichimura-HS-p}.

\begin{theorem}
Let $K$ be a Hilbert-Speiser field of type $C_{p}$ for some odd prime $p$.
If either
\begin{enumerate}
\item $K$ is totally real and $p \geq 7$, or
\item $K$ is totally imaginary and $p > [K:\Q]$,
\end{enumerate}
then $K \cap \Q(\zeta_{p^{n}})=\Q$ for all $n \geq 1$.
\end{theorem}

\section{Realizable Classes}\label{realizable}

We briefly recall the work of McCulloh on realizable classes in the special case of 
cyclic extensions of prime degree (see \cite{mcculloh-abelian} for further details).

Let $K$ be a number field and let $p$ be a prime. 
Let $\Delta \simeq (\Z/p\Z)^{\times}$ be the group of automorphisms of $C_p$. 
Then the locally free class group $\Cl(\mathcal{O}_KC_p)$ is a $\Delta$-module. 
As $L/K$ varies over all tame $C_{p}$-Galois extensions of $K$, 
the class $(\mathcal{O}_L)$ of $\mathcal{O}_L$ varies over a subset $R(\mathcal{O}_KC_p)$ 
of $\Cl(\mathcal{O}_KC_p)$. This subset is in fact a subgroup which can be described explicitly.

Let $\Cl(\mathcal{O}_K)$ denote the ideal class group of $K$ and let $\Cl'(\mathcal{O}_KC_p)$ 
be the kernel of the map $\Cl(\mathcal{O}_KC_p) \longrightarrow \Cl(\mathcal{O}_K)$ induced by augmentation. 
Let $\mathcal{J}$ be the Stickelberger ideal in $\mathbb{Z}\Delta$ (the definition of  $\mathcal{J}$ will be given later).
The key result of relevance to the present paper is that $R(\mathcal{O}_KC_p)$ is the subgroup 
$\Cl'(\mathcal{O}_KC_p)^{\mathcal J}$ of $\Cl(\mathcal{O}_KC_p)$ where 
$\Cl'(\mathcal{O}_KC_p)^{\mathcal J}=\{c^{\alpha}: c\in \Cl'(\mathcal{O}_KC_p), \alpha \in {\mathcal J}\}$.

\section{The proof of Theorem \ref{real-not-HS}}

Let $K$ be a totally real number field and let $p \geq 5$ be prime. 
Let $\mathfrak{p}$ be some prime of $K$ above $p$ and let $e$ denote the 
ramification index of $\mathfrak{p}$ in $K/\Q$. We will assume that $p$ is
ramified in $K/\Q$ and so $\mathfrak{p}$ can be chosen such that $e \geq 2$.
Under these hypotheses we shall show that $K$ is not Hilbert-Speiser
of type $C_{p}$ (note that in the case $p=5$ and $[K(\zeta_{p}):K]=2$ we shall have to assume that $e \geq 3$).

The basic idea of the proof will be to construct certain $\mathcal{O}_{K}$-algebras
$\Gamma$ and $S$ such that $\Gamma \subseteq S$ with 
${S/\Gamma}\simeq {\mathcal{O}_K}/{\mathfrak p}$ as $\mathcal{O}_{K}$-modules. 
Together, $S$ and $\Gamma$ will be used to construct
a non-trivial subgroup of the realizable classes 
$R(\mathcal{O}_KC_p)=\Cl'(\mathcal{O}_KC_p)^{\mathcal J}$ described in Section
\ref{realizable}, thereby giving the desired result.
At all primes ${\mathfrak q}\neq {\mathfrak p}$ of $K$
the completions $S_{\mathfrak q}$ and $\Gamma_{\mathfrak q}$ will be equal, 
so the essential part of the argument will be local at ${\mathfrak p}$.

Let $\phi_p(z)$ be the $p$th cyclotomic polynomial. Then 
$\Gamma:=\mathcal{O}_{K}[z]/(\phi_p(z))$ is an $\mathcal{O}_K$-algebra,
but is a domain if and only if $[K(\zeta_p):K]=p-1$. The group 
$\Delta:=(\Z/p\Z)^{\times}$ acts on $\Gamma$ in the 
following way: to each $\bar{a} \in \Delta$ we associate an automorphism
$\sigma_{a}$ of $\Gamma$ defined by $\sigma_a(z)=z^a$, where 
the image of $z$ in $\Gamma$ is again written $z$.
Let $\omega :\Delta \rightarrow \Z_p^{\times}$ be the Teichm\"uller character, so that 
$\omega(\sigma_a)=\tilde{a}$ where $\tilde{a}^{p-1}=1$, and $\tilde{a}\equiv a\pmod p$.

There exists an element $\lambda$ such that 
${\mathbb Z}_p[\zeta_p]={\mathbb Z}_p[\lambda]$ with 
$\lambda^{p-1}=-p$, $\lambda \equiv 1-\zeta_p \pmod{(1-\zeta_p)^2}$
(see, for example, \cite[Chapter 14, Lemma 3.1]{lang-cyclo-I-and-II}). Furthermore,
$\Delta$ acts on $\lambda^i{\mathbb Z}_p$ through the character $\omega^i$ with 
$i\in\{0,\ldots, p-2\}$.
Note that $\Gamma_{\mathfrak p}=\mathcal{O}_{K_{\mathfrak p}} \otimes_{{\mathbb Z}_{p}} {\mathbb Z}_p[\zeta_p]$, for which 
$\{ 1, \lambda, \lambda^{2}, \ldots \lambda^{p-2}\}$ is an 
$\mathcal{O}_{K_{\mathfrak{p}}}$-basis .
Let $\pi$ denote a parameter of $\mathcal{O}_{K_{\mathfrak p}} $ and define
the element $x : = {\frac{1}{\pi}}\otimes{\lambda^{p-2}}$ in 
$K_{\mathfrak p} \otimes_{\Q_{p}}\Q_{p}(\zeta_{p})=K_{\mathfrak p}\Gamma$ (we will abuse notation and write $x=\frac{\lambda^{p-2}}{\pi}$).

\begin{lemma}\label{elements-lemma}
We have $x^2, x^3, \lambda x, \pi x \in {\Gamma}_{\mathfrak p}$.
\end{lemma}

\begin{proof}
Since $e \geq 2$, we have $\frac{p}{\pi^{2}} \in \mathcal{O}_{K_{\mathfrak{p}}}$.
Hence
\[
x^2=\frac{\lambda^{2p-4}}{\pi^2}=\frac{-p\lambda^{p-3}}{\pi^2}
\quad \textrm{ and } \quad
x^3=\frac{(-p)^2\lambda^{p-4}}{\pi^3}
\]
are both in ${\Gamma}_{\mathfrak p}$ (we have used that $p \geq 5$ here). 
Furthermore, it is clear that
\[
\lambda x = \frac{\lambda^{p-1}}{\pi} = \frac{-p}{\pi}
\quad
\textrm{and}
\quad
\pi x = \lambda^{p-2}
\]
are both in ${\Gamma}_{\mathfrak p}$.
\end{proof}

We shall now consider three cases, the first two of which overlap.

\subsection{The case $[K(\zeta_{p}):K]>2$}

A consequence of Lemma \ref{elements-lemma} is that the 
$\mathcal{O}_{K_{\mathfrak{p}}}$-module 
$T:= \Gamma_{\mathfrak p}+x\mathcal{O}_{K_{\mathfrak p}} $ is in fact an 
$\mathcal{O}_{K_{\mathfrak p}}$-algebra. Furthermore, we have
\[
\pi T = \pi \Gamma_{\mathfrak p}+\lambda^{p-2}\mathcal{O}_{K_{\mathfrak p}}
\subseteq \Gamma_{\mathfrak p} \subseteq T
\]
since $\lambda^{p-2}$ is part of an $\mathcal{O}_{K_{\mathfrak p}}$-basis of $\Gamma_{\mathfrak p}$.
We now let $S$ be the 
$\mathcal{O}_K$-order defined by
\begin{eqnarray*}
S_{\mathfrak q} & = & \Gamma_{\mathfrak q} \quad ({\mathfrak q}\neq {\mathfrak p}); \\
S_{\mathfrak p} & = & T.
\end{eqnarray*}
We find that $\Gamma \subseteq S$ and $\pi S \subseteq \Gamma$ 
(note that we have abused notation in the obvious way here). Furthermore, the ring 
$\bar{S}:=S/\pi S$ is isomorphic to 
$S_{\mathfrak{p}}/\pi S_{\mathfrak{p}}=T/\pi T$.
Let $\bar{\Gamma}$ be the image of $\Gamma$ under the canonical map 
$S\rightarrow \bar{S}$. We have a Milnor square 
\[
\begin{array}{ccc}
\Gamma & \hookrightarrow & S \\
\downarrow &  & \downarrow \\
\bar{\Gamma}& \hookrightarrow &\bar{S}
\end{array}
\]
where the horizontal arrows are the natural inclusions and the vertical arrows are the natural projections (note that this is a special case of a fiber product). 
Note that we have $\delta(\lambda^{p-2}) \equiv \omega^{p-2}(\delta)\lambda^{p-2}$ modulo $p{\mathbb Z}_p[\zeta_p]$ for every $\delta\in \Delta$, so $\delta(x) \in x+\Gamma \subset S$. Hence $\Delta$ acts on $S$ and so acts on each of the rings in the Milnor square.
By \cite[p.242]{curtis-reiner-II} we have the following exact sequence
\[
\K_{1}(S) \times \K_{1}(\bar{\Gamma})
\longrightarrow 
\K_{1}(\bar{S})
\longrightarrow
\Cl(\Gamma)
\longrightarrow
\Cl(S)
\longrightarrow 0.
\]
As all the rings above are commutative, this becomes
\[
S^{\times} \times \bar{\Gamma}^{\times}
\longrightarrow 
\bar{S}^{\times}
\longrightarrow
\Cl(\Gamma)
\longrightarrow
\Cl(S)
\longrightarrow 0.
\]
Hence we have an embedding of $\Delta$-modules
\[
N:=\frac{\bar{S}^{\times}}{\bar{\Gamma}^{\times} \cdot {\rm im}(S^{\times})}
\hookrightarrow
\Cl(\Gamma),
\]
where ${\rm im}(S^{\times})$ is the image of $S^{\times}$ under the map 
$S\rightarrow \bar{S}$.

For every $\Delta$-module $X$, let $X^-$ and $X^{\omega^{-1}}$ denote the
minus part and $\omega^{-1}$-part of $\mathbb{Z}_p \otimes_{\mathbb{Z}} X$, respectively.
Then $X^{\omega^{-1}}\subseteq X^-$. We will show that $N^{\omega^{-1}}$ contains a submodule $M$ of order $p$. Note that by the definition of $x$ and the action of $\Delta$,
we have $x \in S^{\omega^{-1}}$. We define $\bar{x} \in \bar{S}$ to be the image of $x \in T$ under the natural projection $T \rightarrow T/\pi T \simeq \bar{S}$ and note that 
$\bar{x} \in \bar{S}^{\omega^{-1}}$.

Let $[{\rm exp}](z):=\sum_{i=0}^{p-1}\frac{1}{i!}z^i$ denote the truncated exponential series.  
Whenever the ideal $(a,b)$ generated by $a$ and $b$ satisfies $(a, b)^p=0$, we have $[{\rm exp}](a+b)=[{\rm exp}](a)\cdot [{\rm exp}](b)$ (see the proof of 
\cite[$p$-elementary group schemes---constructions and Raynaud's theory, Remark 1.1]{Hopf-poly-Raynaud}). Let $y:=[{\rm exp}](\bar{x})\in \bar{S}$.
Since $y^p=[{\rm exp}](p\bar{x})=[{\rm exp}](0)=1$, we have $y\in {\bar{S}}^{\times}$. 
We note that $y\notin \bar{\Gamma}$ (the summand with $i=1$ is $\bar x$ and hence plainly outside $\bar{\Gamma}$, and all other summands are in $\bar{\Gamma}$ by Lemma \ref{elements-lemma}). Moreover, as [exp] is compatible with the $\Delta$-action, we have $y\in (\bar{S}^{\times})^{\omega^{-1}} = \bar{S}^{\times} \cap \bar{S}^{\omega^{-1}}$. 

\begin{lemma}\label{omega-equality}
We have 
$(\bar{\Gamma}^{\times}\cdot {\rm im}(S^{\times}))^{\omega^{-1}}
=({\bar{\Gamma}^{\times}})^{\omega^{-1}}$.
\end{lemma} 

\begin{proof}
Let ${\mathfrak M}$ denote the maximal order in $KS=K\Gamma$. Then 
${\mathfrak M}={\rm ind}_{\Delta_0}^{\Delta}\mathcal{O}_{K(\zeta_p)}$ with 
$\Delta_0=\Gal(K(\zeta_p)/K)$. 
We consider ${S^{\times}}^-\subseteq {{\mathfrak M}^{\times}}^-$; since $K$ is totally real, complex conjugation $j \in \Delta_{0}$ acts on each factor of $\mathfrak{M}$ separately, and we see that $\mathcal{O}_{K(\zeta_p)}^{\times -}$ is the multiplicative group of roots of unity $\langle \zeta_{{p}^{f}} \rangle$ for some $f \geq 1$ (see \cite[Theorem 4.12]{wash}).
Hence ${\mathfrak{M}^{\times}}^-={\rm ind}_{\Delta_0}^{\Delta}\langle \zeta_{p^{f}} \rangle$.

Suppose that $f=1$. Then $\Delta_0$ acts on $\zeta_p$ via $\omega \vert_{\Delta_0}$, 
and from the Frobenius reciprocity theorem one obtains that 
${\rm ind}_{\Delta_0}^{\Delta}\langle \zeta_p \rangle$ has non-trivial
$\omega^{-1}$-part if and only if ${\omega^{-1}}\vert_{\Delta_0}=\omega \vert_{\Delta_0}$, that is, if and only if $\omega^2$ is trivial on $\Delta_0$. But this is not the case since 
$[K(\zeta_{p}):K]=|\Delta_0|>2$ by hypothesis. 
Now suppose $f>1$. Then considering short exact sequence
\[
1 \longrightarrow {\rm ind}_{\Delta_0}^{\Delta}\langle \zeta_{p^{f-1}} \rangle
\longrightarrow {\rm ind}_{\Delta_0}^{\Delta}\langle \zeta_{p^f} \rangle
\longrightarrow {\rm ind}_{\Delta_0}^{\Delta}\langle \zeta_{p} \rangle
\longrightarrow 1,
\]
we see that the middle term has trivial $\omega^{-1}$-part if and only if 
the same is true of both the outer terms. It now follows by induction on $f$ that 
${\mathfrak{M}^{\times}}^-={\rm ind}_{\Delta_0}^{\Delta}\langle \zeta_{p^{f}} \rangle$ 
has trivial $\omega^{-1}$-part. Hence $(S^{\times})^{\omega^{-1}}$ is trivial, 
and the lemma is proved.
\end{proof}

Let $\bar{y}$ denote the projection of $y$ to $N$. If $\bar{y}$ were trivial
in $N$, then $y$ would have to be in $(\bar{\Gamma}^{\times}\cdot {\rm im}(S^{\times}))^{\omega^{-1}}=({\bar{\Gamma}^{\times}})^{\omega^{-1}}$. However, we have already
noted that $y$ is not even in $\bar{\Gamma}$. Hence $M:= \langle \bar{y} \rangle$ is a 
non-trivial $\Delta$-submodule of $N$ with $M^{\omega^{-1}}=M$.

\subsection{The case $e \geq 4$}

Let $x_{1}=x=\frac{\lambda^{p-2}}{\pi} = \frac{1}{\pi} \otimes \lambda^{p-2}$ be as above and define $x_{2}= \frac{\lambda^{p-2}}{\pi^{2}} = \frac{1}{\pi^{2}} \otimes \lambda^{p-2}$.

\begin{lemma}\label{elements-lemma2}
We have $x_{2}^2, x_{2}^3, \lambda x_{2}, \pi^2 x_{2}, x_{1}x_{2}, x_{1}^{2}x_{2}, 
x_{1}x_{2}^{2}  \in {\Gamma}_{\mathfrak p}$.
\end{lemma}

\begin{proof}
We use that $p \geq 5$ without further mention.
Since $e \geq 4$, we have $\frac{p}{\pi^{4}} \in \mathcal{O}_{K_\mathfrak{p}}$.
Hence
\[
x_{2}^2=\frac{\lambda^{2p-4}}{\pi^4}=\frac{-p\lambda^{p-3}}{\pi^4}
\quad \textrm{ and } \quad
x_{2}^3=\frac{p^2\lambda^{p-4}}{\pi^6}
\]
are both in ${\Gamma}_{\mathfrak p}$.
Furthermore, it is clear that
\[
\lambda x_{2} = \frac{\lambda^{p-1}}{\pi^{2}} = \frac{-p}{\pi^{2}}
\quad
\textrm{and}
\quad
\pi^{2} x_{2} = \lambda^{p-2}
\]
are both in ${\Gamma}_{\mathfrak p}$.
Finally,
\begin{align*}
x_{1}x_{2} = \frac{\lambda^{2p-4}}{\pi^{3}} &= \frac{(-p)\lambda^{p-3}}{\pi^{3}},
\quad
x_{1}^{2}x_{2} = \frac{\lambda^{3p-6}}{\pi^{4}} = \frac{(-p)^{2}\lambda^{p-4}}{\pi^{4}},
\\
\textrm{and } \quad x_{1}x_{2}^{2} &= \frac{\lambda^{3p-6}}{\pi^{5}} = \frac{(-p)^{2}\lambda^{p-4}}{\pi^{5}}
\end{align*}
are all in $\Gamma_{\mathfrak{p}}$.
\end{proof}

A consequence of Lemmas \ref{elements-lemma} and \ref{elements-lemma2} is that
the $\mathcal{O}_{K_\mathfrak{p}}$-module 
$T:=\Gamma_{\mathfrak{p}}+x_{1}\mathcal{O}_{K_\mathfrak{p}}+x_{2}\mathcal{O}_{K_{\mathfrak{p}}}$ is in fact
an $\mathcal{O}_{K_{\mathfrak{p}}}$-algebra. Furthermore, we have
\[
\pi^{2}T = \pi^{2}\Gamma_{\mathfrak{p}} + \pi \lambda^{p-2}\mathcal{O}_{K_\mathfrak{p}}
+ \lambda^{p-2}\mathcal{O}_{K_\mathfrak{p}} 
= \pi^{2}\Gamma_{\mathfrak{p}} + \lambda^{p-2}\mathcal{O}_{K_\mathfrak{p}} 
\subseteq \Gamma_{\mathfrak{p}} \subseteq T
\]
since $\lambda^{p-2}$ is part of an $\mathcal{O}_{K_{\mathfrak{p}}}$-basis of $\Gamma_{\mathfrak{p}}$. 
We now let $S$ be the 
$\mathcal{O}_K$-order defined by
\begin{eqnarray*}
S_{\mathfrak q} & = & \Gamma_{\mathfrak q} \quad ({\mathfrak q}\neq {\mathfrak p}); \\
S_{\mathfrak p} & = & T.
\end{eqnarray*}
Then $\Gamma \subseteq S$ and $\pi^{2}S \subseteq \Gamma$.  
The same argument as in the previous case gives
an embedding of $\Delta$-modules
\[
N:=\frac{\bar{S}^{\times}}{\bar{\Gamma}^{\times} \cdot {\rm im}(S^{\times})}
\hookrightarrow
\Cl(\Gamma).
\]
By definition of $x_{1}, x_{2}$ and the action of $\Delta$, we have 
$x_{1}, x_{2} \in S^{\omega^{-1}}$. Let $y_{1} = [{\rm exp}](\bar{x}_{1})$, 
$y_{2} = [{\rm exp}](\bar{x}_{2}) \in \bar{S}=S/\pi^2 S$. As in the previous case,
both $y_{1}, y_{2}$ are elements of order $p$ in $(\bar{S}^{\times})^{\omega^{-1}}$.

\begin{lemma}\label{cyclic}
$(S^{\times})^{\omega^{-1}}$ is cyclic.
\end{lemma}

\begin{proof}
Let ${\mathfrak M}$ denote the maximal order in $KS=K\Gamma$. 
As $S \subseteq \mathfrak{M}$, we have 
${(S^{\times}})^{\omega^{-1}} \subseteq {({\mathfrak M}^{\times}})^{\omega^{-1}}$,
and so it suffices to show that $({{\mathfrak M}^{\times}})^{\omega^{-1}}$ is a cyclic group.

By the same argument as for Lemma \ref{omega-equality}, we obtain ${\mathfrak{M}^{\times}}^-={\rm ind}_{\Delta_0}^{\Delta}\langle \zeta_{p^{f}} \rangle$.
Furthermore, as $\omega^{-1}$ is an odd character, we have 
$({\mathfrak M}^{\times})^{\omega^{-1}} \subseteq {{\mathfrak M}^{\times}}^-$.
Now $\langle \zeta_{p^{f}} \rangle$ is trivially cyclic as a $\Z_{p}[\Delta_{0}]$-module;
hence 
\[ 
{\mathfrak{M}^{\times}}^-={\rm ind}_{\Delta_0}^{\Delta}\langle \zeta_{p^{f}} \rangle = 
\Z_{p}[\Delta] \otimes_{\Z_{p}[\Delta_{0}]}\langle \zeta_{p^{f}} \rangle
\]
is cyclic as a $\Z_{p}[\Delta]$-module. Thus $({\mathfrak{M}^{\times}})^{\omega^{-1}} =
\Z_{p}(\omega^{-1}) \otimes_{\Z_{p}[\Delta]} {\mathfrak{M}^{\times}}^-$
is cyclic as a $\Z_{p}(\omega^{-1})$-module, where $\Z_{p}(\omega^{-1})$ is the ring extension of $\Z_{p}$ obtained by adjointing the image of $\omega^{-1}$. However, $\omega^{-1}$ takes its values in $\Z_{p}^{\times}$, and so $\Z_{p}(\omega^{-1})=\Z_{p}$. Therefore
$({\mathfrak M}^{\times})^{\omega^{-1}}$ is cyclic as a $\Z_{p}$-module, and hence is cyclic 
as a group.
\end{proof}

Let $\tilde{y}_{1}, \tilde{y}_{2} \in 
(\bar{S}^{\times}/\bar{\Gamma}^{\times})^{\omega^{-1}}$ be the images
of $y_{1}, y_{2}$ under the natural projection. Since $y_{1}, y_{2} \notin \bar{\Gamma}$ 
(the summand with $i=1$ is outside $\bar{\Gamma}$ and all others are in $\bar{\Gamma}$ by Lemmas \ref{elements-lemma} and \ref{elements-lemma2}), $\tilde{y}_{1}, \tilde{y}_{2}$ are also each 
of order $p$. 
Suppose that 
$\tilde{y}_{1}^{k_{1}}\tilde{y}_{2}^{k_{2}}$ 
is trivial for some $k_{1}, k_{2} \in \{ 1, 2, \ldots, p-1\}$. 
This would mean that $[{\rm exp}](k_{1}\bar{x}_{1}+k_{2}\bar{x}_{2})$ is in $\bar \Gamma$.
By virtue of Lemma \ref{elements-lemma2}, we  have $[{\rm exp}](k_{1}\bar{x}_{1}+k_{2}\bar{x}_{2})
\equiv 1+k_{1}\bar{x}_{1}+k_{2}\bar{x}_{2}$ modulo $\bar{\Gamma}$. Therefore we would obtain
$k_{1}\bar{x}_{1}+k_{2}\bar{x}_{2} \in \bar{\Gamma}$ and so
$k_{1}x_{1}+k_{2}x_{2} \in \Gamma_{\mathfrak{p}}$, which is impossible. 
Hence the subgroup 
$\langle \tilde{y}_{1}, \tilde{y}_{2} \rangle \simeq \langle y_{1}, y_{2} 
\rangle \simeq \Z/p\Z \times \Z/p\Z$ is non-cyclic. 
Let $\bar{y}_{1}, \bar{y}_{2}$ be the projections of $\tilde{y}_{1}, \tilde{y}_{2}$ to $N$
and let $M := \langle \bar{y}_{1}, \bar{y}_{2} \rangle \subseteq N^{\omega^{-1}}$.
Note that $M$ is non-trivial because $({\rm im}(S^{\times}))^{\omega^{-1}}$ is cyclic by Lemma
\ref{cyclic}, but $\langle \tilde{y}_{1}, \tilde{y}_{2} \rangle$ is non-cyclic. Hence $M$ is a non-trivial
$\Delta$-submodule of $N$ with $M^{\omega^{-1}}=M$.

\subsection{The case $[K(\zeta_{p}):K]=2$ and  $e=2$ or $3$}

Note that the condition $[K(\zeta_{p}):K]=2$ implies that $\frac{p-1}{2}$ divides $e$.
Hence we are reduced to considering the cases $p=5$ and $p=7$
(since $p \geq 11$ forces $e \geq 5$). If $p=5$, then $e$ must be even
and so in fact $e=2$. However, this case is excluded by hypothesis. If $p=7$, then
we must have $e=3$. In this case, we let $x_{1}=\frac{\lambda^{5}}{\pi}$, 
$x_{2}=\frac{\lambda^{5}}{\pi^{2}}$ and $x_{3}=\frac{\lambda^{4}}{\pi}$. It is straightforward 
to check that the $\mathcal{O}_{K_{\mathfrak{p}}}$-module 
$T:=\Gamma_{\mathfrak{p}}+x_{1}\mathcal{O}_{K_{\mathfrak{p}}}+x_{2}\mathcal{O}_{K_{\mathfrak{p}}}
+x_{3}\mathcal{O}_{K_{\mathfrak{p}}}$ is in fact an $\mathcal{O}_{K_{\mathfrak{p}}}$-algebra.
The result is then given by a slight variant of the proof of
the previous case (note that $x_{1}, x_{2} \in S^{\omega^{-1}}$ but $x_{3} \notin S^{\omega^{-1}}$).

\subsection{The proof of Theorem \ref{real-not-HS}}

In each of the above cases, we have shown that there exists a non-trivial $\Delta$-submodule
$M$ of $\Cl(\Gamma)$ such that $M^{\omega^{-1}}=M$.

\begin{proof}[Proof of Theorem \ref{real-not-HS}]

Recall that the Stickelberger ideal is defined to be ${\mathcal J}={\mathbb Z}\Delta\cap \theta\cdot {\mathbb Z}\Delta={\rm Ann}_{\Delta}(\langle \zeta_p \rangle ) \cdot \theta$ where $\theta$ is the Stickelberger element $\frac{1}{p}\sum_{j=1}^{p-2}j\sigma_j^{-1}$. Let ${\mathcal J}_p\subseteq {\mathbb Z}_p\Delta$ be the $p$-completion  of $\mathcal J$. Then $\omega^{-1}({\mathcal J}_p)=\omega^{-1}({\rm Ann}_{\Delta}(\langle \zeta_p \rangle))\cdot {\omega}^{-1}\theta$. The second factor of the last expression is the generalized Bernoulli number $B_{1, \omega}$. Since $p\geq 5$, the first factor  is ${\mathbb Z}_p$. 
By \cite[Corollary 5.15]{wash} we have
$B_{1, \omega}\equiv \frac{B_2}{2}=\frac{1}{12} \pmod p$. Hence, $M^{\mathcal J}=M^{\omega^{-1}({\mathcal J}_p)}=M^{{\mathbb Z}_p}=M$ and therefore
$\Cl(\Gamma)^{\mathcal{J}} \neq 0$.

Let $\Sigma$ denote the sum of the elements of $C_p$. 
Consider the following Milnor square
\[
\begin{array}{ccc}
\mathcal{O}_KC_p &\stackrel{\alpha}{ \longrightarrow} & \mathcal{O}_KC_p/\mathcal{O}_K\Sigma =: \Lambda \\
\downarrow {\scriptstyle \beta} &  & \downarrow {\scriptstyle \gamma}  \\
\mathcal{O}_K&\longrightarrow &\mathcal{O}_K/p\mathcal{O}_K,
\end{array}
\]
where the horizontal maps are the natural projections, $\beta$ is the augmentation
map, and $\gamma$ is the map induced by augmentation. The resulting map 
\[
\Cl(\mathcal{O}_KC_p)\stackrel{(\alpha, \beta)}{\longrightarrow}\Cl(\Lambda)\times \Cl(\mathcal{O}_K)
\]
is surjective (see, for instance, \cite[Corollary 49.28]{curtis-reiner-II}). 
It follows immediately that
\[
\Cl'(\mathcal{O}_KC_p) \longrightarrow \Cl(\Lambda)
\]
is surjective. However, $\Cl(\Lambda) \simeq \Cl(\Gamma)$ since
$\Lambda \simeq \Gamma$, and so $\Cl(\Lambda)^{\mathcal{J}} \neq 0$.
Therefore $\Cl'(\mathcal{O}_KC_p)^{\mathcal J}=R(\mathcal{O}_KC_p) \neq 0$,
and so $K$ is not a Hilbert-Speiser field of type $C_p$.
\end{proof}

\section{Acknowledgments}

The authors are grateful to the Deutscher Akademischer Austausch Dienst (German Academic Exchange Service) for a grant allowing the second named author to visit the first for the 2006-07 academic year, thus making this collaboration possible. 

The authors are indebted to James E. Carter for suggesting the original problem,
to Nigel P. Byott for pointing out an error in an earlier version of this paper, and to
both the aforementioned and the referee for several helpful comments and suggestions. 

\bibliography{not-p-HS-Bib}{}

\newcommand{\etalchar}[1]{$^{#1}$}
\providecommand{\bysame}{\leavevmode\hbox to3em{\hrulefill}\thinspace}
\providecommand{\MR}{\relax\ifhmode\unskip\space\fi MR }
\providecommand{\MRhref}[2]{%
  \href{http://www.ams.org/mathscinet-getitem?mr=#1}{#2}
}
\providecommand{\href}[2]{#2}
\begin{thebibliography}{CGM{\etalchar{+}}98}

\bibitem[Car03]{carter-NIB-quad-cubic}
James~E. Carter, \emph{Normal integral bases in quadratic and cyclic cubic
  extensions of quadratic fields}, Arch. Math. (Basel) \textbf{81} (2003),
  no.~3, 266--271. \MR{2013257 (2004k:11167)}

\bibitem[Car04]{carter-erratum}
\bysame, \emph{Erratum to: ``{N}ormal integral bases in quadratic and cyclic
  cubic extensions of quadratic fields'' [{A}rch. {M}ath. ({B}asel) {\bf 81}
  (2003), no. 3, 266--271; {MR} 2013257]}, Arch. Math. (Basel) \textbf{83}
  (2004), no.~6, vi--vii. \MR{2191482 (2006h:11122)}

\bibitem[CGM{\etalchar{+}}98]{Hopf-poly-Raynaud}
Lindsay~N. Childs, Cornelius Greither, David~J. Moss, Jim Sauerberg, and Karl
  Zimmermann, \emph{Hopf algebras, polynomial formal groups, and {R}aynaud
  orders}, Mem. Amer. Math. Soc. \textbf{136} (1998), no.~651, viii+118.
  \MR{1629460 (2000c:14067)}

\bibitem[CR87]{curtis-reiner-II}
Charles~W. Curtis and Irving Reiner, \emph{Methods of representation theory.
  {V}ol. {II}}, Pure and Applied Mathematics (New York), John Wiley \& Sons
  Inc., New York, 1987, With applications to finite groups and orders, A
  Wiley-Interscience Publication. \MR{892316 (88f:20002)}

\bibitem[GRRS99]{grrs}
Cornelius Greither, Daniel~R. Replogle, Karl Rubin, and Anupam Srivastav,
  \emph{Swan modules and {H}ilbert-{S}peiser number fields}, J. Number Theory
  \textbf{79} (1999), no.~1, 164--173. \MR{1718724 (2000m:11111)}

\bibitem[Her05]{herreng-HS}
Thomas Herreng, \emph{Sur les corps de {H}ilbert-{S}peiser}, J. Th\'eor.
  Nombres Bordeaux \textbf{17} (2005), no.~3, 767--778. \MR{2212124
  (2006k:11220)}

\bibitem[Ich02]{ichimura-kummer-prime-V}
Humio Ichimura, \emph{Note on the ring of integers of a {K}ummer extension of
  prime degree. {V}}, Proc. Japan Acad. Ser. A Math. Sci. \textbf{78} (2002),
  no.~6, 76--79. \MR{1913934 (2003g:11120)}

\bibitem[Ich04]{ichimura-NIB-rayclassgroups}
\bysame, \emph{Normal integral bases and ray class groups}, Acta Arith.
  \textbf{114} (2004), no.~1, 71--85. \MR{2067873 (2005c:11138)}

\bibitem[Ich07a]{Ichimura-HS-p}
\bysame, \emph{Note on {H}ilbert-{S}peiser number fields at a prime {$p$}},
  Yokohama Math. J. \textbf{54} (2007), no.~1, 45--53. \MR{MR2394019}

\bibitem[Ich07b]{Ichimura-imag-quad}
\bysame, \emph{Note on imaginary quadratic fields satisfying the
  {H}ilbert-{S}peiser condition at a prime {$p$}}, Proc. Japan Acad. Ser. A
  Math. Sci. \textbf{83} (2007), no.~6, 88--91. \MR{2355504}

\bibitem[IST07]{Ichimura-Sumida-Takahashi-HS-imag-quad}
Humio Ichimura and Hiroki Sumida-Takahashi, \emph{Imaginary quadratic fields
  satisfying the {H}ilbert-{S}peiser type condition for a small prime {$p$}},
  Acta Arith. \textbf{127} (2007), no.~2, 179--191. \MR{2289983 (2008c:11152)}

\bibitem[Lan90]{lang-cyclo-I-and-II}
Serge Lang, \emph{Cyclotomic fields {I} and {II}}, second ed., Graduate Texts
  in Mathematics, vol. 121, Springer-Verlag, New York, 1990, With an appendix
  by Karl Rubin. \MR{1029028 (91c:11001)}

\bibitem[McC83]{mcculloh-abelian}
Leon~R. McCulloh, \emph{Galois module structure of elementary abelian
  extensions}, J. Algebra \textbf{82} (1983), no.~1, 102--134. \MR{701039
  (85d:11093)}

\bibitem[Was97]{wash}
Lawrence~C. Washington, \emph{Introduction to cyclotomic fields}, second ed.,
  Graduate Texts in Mathematics, vol.~83, Springer-Verlag, New York, 1997.
  \MR{1421575 (97h:11130)}

\end{thebibliography}
\bibliographystyle{amsalpha}

 \end{document}